\documentclass[12pt,oneside]{amsart}
\usepackage[a4paper, total={4.5in, 7.5in}]{geometry}

\usepackage{amsthm,amsmath,amssymb}
\usepackage{mathrsfs}
\usepackage{enumerate}
\usepackage{amsfonts}
\usepackage{verbatim}
\usepackage{amsbsy}
\usepackage{amsmath}
\usepackage{amssymb}
\usepackage{MnSymbol}
\usepackage{mathrsfs}
\usepackage{xcolor}
\usepackage{cite}

\theoremstyle{definition}

\theoremstyle{remark}

\newtheorem{dfn}{Definition}[section]
\newtheorem{thm}{Theorem}[section]
\newtheorem{prop}{Proposition}[section]
\newtheorem{lem}{Lemma}[section]
\newtheorem{cor}{Corollary}[section]
\newtheorem{rem}{Remark}[section]
\newtheorem{eg}{Example}[section]

\DeclareMathOperator{\FR}{FR}

\DeclareMathOperator{\id}{id}
\DeclareMathOperator{\fr}{FR}
\DeclareMathOperator{\ot}{OT}
\DeclareMathOperator{\cod}{Cod}
\DeclareMathOperator{\rn}{Rn}
\DeclareMathOperator{\dom}{Dom}

\begin{document}

\title[Heterogeneous Ramsey Algebras]{Heterogeneous Ramsey Algebras and Classification of Ramsey Vector Spaces}
\author{Zu Yao Teoh}
\address{School of Mathematical Sciences\\
Universiti Sains Malaysia\\
11800 USM, Malaysia}
\email{teohzuyao@gmail.com}
\author{Wen Chean Teh*}\thanks{*Corresponding author.}
\address{School of Mathematical Sciences\\
Universiti Sains Malaysia\\
11800 USM, Malaysia}
\email{dasmenteh@usm.my}

\begin{abstract}
Carlson introduced the notion of a Ramsey space as a generalization to the Ellentuck space. When a Ramsey space is induced by an algebra, Carlson
suggested a study of its purely combinatorial version now called Ramsey algebra. Some basic results for homogeneous algebras have been obtained. In
this paper, we introduce the notion of a Ramsey algebra for heterogeneous algebras and derive some basic results. Then, we study the Ramsey-algebraic
properties of vector spaces.
\end{abstract}

\maketitle

\section{Introduction}\label{intro}
The notion of a Ramsey space was introduced by Carlson in \cite{carlson1988some}. In the modern literature, this notion is referred to as a topological
Ramsey space. The class of topological Ramsey spaces induced by algebras has been singled out to be studied from a purely combinatorial point of view.
An algebra is a structure consisting of a family of sets, called the domain of the algebra, and a family of operations on these sets. Hindman's Theorem
states that for each positive integer $r$ and each coloring $c:\mathbb{Z}^+\to\{1, \ldots, r\}$, there exists an infinite $S\subseteq\mathbb{Z}^+$ such
that $c$ is constant on the set $\left\{\sum_{i\in F} i:F\subseteq S, F\:\text{is finite}\right\}$. The same theorem can be cast in terms of finite
subsets of the natural numbers \cite{milliken1975ramsey}. Considering questions about infinite sets along the same line, a result of Erd\"{o}s-Rado
\cite{erdos1952combinatorial} shows that not all sets of reals\footnote{We are adopting the convention in Set Theory, where the infinite subsets of
natural numbers are identified with the real numbers.} possess similar property. Hence, it is natural to ask which among those definable
sets--specifically, sets belonging to the $\sigma$-algebra generated by the usual topology on the reals--have the analogous property. Galvin-Prikry
\cite{GP73} showed that the Borel sets do, whereas, using the method of forcing, Silver \cite{silver1970every} showed that the analytic sets have the
desired property.

Ellentuck in \cite{ellentuck1974new} showed that the result of Silver could also be proved using topological formulation without having to appeal to
metamathematical methods. In his proof, Ellentuck introduced what is to be called the Ellentuck topology. A generalization of the Ellentuck topology is
then introduced by Carlson in \cite{carlson1988some}, which he called (topological) Ramsey space. Carlson showed in the same paper that classical
results such as the Hales-Jewett Theorem and Hindman's Theorem can be derived as corollaries to his results on the Ramsey space of variable words.

If one is to ask whether a given structure is a Ramsey space, the definition requires that one checks some topological properties of the space. However,
Carlson's abstract version of Ellentuck's Theorem (cf. \cite{carlson1988some}) turns such topological question into a combinatorial one. Carlson had, in
fact, suggested that one could embark on a specifically combinatorial study of spaces induced by algebras; this study was then initiated by Teh [cf.
\cite{teh2016ramsey} \& \cite{wcT13a}] and the topic became known as Ramsey algebra.

There is already quite a number of existing Ramsey algebraic results on algebras commonly encountered in the mathematical literature; in particular, a
result on infinite integral domains, which will be applied to our work in this paper, can be found in \cite{teh2016ramsey}. Other results including the
existence of idempotent ultrafilters can be found in \cite{wcT13b} and \cite{wcTIUF}. In this paper, we extend the study of Ramsey algebras to
many-sorted structures. We will first formulate a notion of Ramsey algebra for many-sorted algebras, then we will derive some results pertaining to
vector spaces. A latter section is then dedicated to the study of the homogeneous structure over a family of functions induced by the nonzero scalars of
a vector space.

\section{Preliminary}\label{prelim}
We begin with the notion of an \emph{algebra}. From a logical point of view, an algebra is a structure interpreting a many-sorted, purely functional
language. That is, an algebra $\mathcal{A}$ consists of a family $(A_\xi)_{\xi\in I}$ of sets and a family $\mathcal{F}$ of functions, where each
$f\in\mathcal{F}$ is a function with domain some Cartesian product $A_{\xi_1}\times\cdots\times A_{\xi_n}$ ($n\in\omega$) of some members of
$(A_\xi)_{\xi\in I}$ and codomain some member $A$ of $(A_\xi)_{\xi\in I}$. We write $\mathcal{A}=((A_\xi)_{\xi\in I}, \mathcal{F})$ for the algebra just
described. The family of sets pertaining to a given algebra is called the \emph{universe} or \emph{domain} of the algebra and each set in the family a
\emph{phylum} as per \cite{birkhoff1970heterogeneous}. Throughout this paper, we reserve the symbol $I$ for the index set of the algebra under
discussion.

When $I$ is a singleton, the algebra is said to be \emph{homogeneous}. Examples of homogeneous algebras are groups, rings, Boolean algebras, etc.. For a
homogeneous algebra, the notation $((A_\xi)_{\xi\in I}, \mathcal{F})$ will be simplified to $(A, \mathcal{F})$, where $A$ is the only phylum of the
domain. A vector space is a natural example of a heterogeneous algebra consisting of two phyla.

The set of natural numbers, inclusive of $0$, is denoted by $\omega$.  If $f$ is a function, $\dom(f)$ denotes the domain of $f$, $\cod(f)$ denotes the
codomain of $f$, and $\rn(f)$ denotes the range of $f$. We will identify any $n$-tuple $(x_1, \ldots, x_n)$ with the finite sequence $\langle x_1,
\ldots, x_n\rangle$ and, if $\bar{x}$ denote the $n$-tuple, $\vec{x}$ will denote the associated sequence, and vice versa. If $\vec{e}$ is an infinite
sequence, the notation $\vec{e}-n$ denotes the tail $\langle\vec{e}(n), \vec{e}(n+1), \ldots\rangle$ and $\vec{e}\upharpoonright n$ denotes the initial
segment $\langle\vec{e}(0), \ldots, \vec{e}(n)\rangle$. If $\sigma$ is a finite sequence, then $|\sigma|$ denotes the length of the sequence. If $A$ is
a nonempty set, an infinite sequence of $A$ is an element of ${^\omega}A$. If $D$ is a set, then $\id_D$ denotes the identity function on $D$.

While the family of phyla $(A_\xi)_{\xi\in I}$ in a given heterogeneous algebra $((A_\xi)_{\xi\in I}, \mathcal{F})$ need not be pairwise disjoint in
general, we shall always assume that the family is indeed pairwise disjoint in the study of heterogeneous Ramsey algebra following the assumption made
by Carlson in \cite{carlson1988some}.

\section{The Basic Notions of Heterogeneous Ramsey Algebras}\label{basics}
The study of Ramsey algebra is a fairly recent endeavor. In this section, we define the notion of a Ramsey algebra and the necessary definitions leading
up to it. We will begin with the notion of an orderly composition due to Carlson.

\begin{dfn}[Orderly Composition \& Orderly Terms]\label{ordcomp}
Let $\mathcal{F}$ be a family of operations on $(A_\xi)_{\xi\in I}$. An $n$-ary function $f$ is called an \emph{orderly composition} of $\mathcal{F}$ if
there exists $h_1, \ldots, h_k, g\in\mathcal{F}$ such that
\begin{enumerate}
\item $g$ is a $k$-ary function,
\item $h_j$ is an $n_j$-ary function for each $j\in\{1, \ldots, k\}$,
\item $\sum_{j=1}^k n_j=n$, and
\item if $\bar{x}_1=(x_1, \ldots, x_{n_1})$ and $\bar{x}_j=\left(x_{\sum_{i=1}^{j-1}n_i+1}, \ldots, x_{\sum_{i=1}^jn_i}\right)$ for each $j=\{2$,
    $\ldots$, $k\}$, then $f(x_1, \ldots, x_n)=g(h_1(\bar{x}_1), \ldots, h_k(\bar{x}_k))$.
\end{enumerate}

The collection $\ot(\mathcal{F})$ of \emph{orderly terms} over $\mathcal{F}$ is the \emph{smallest} collection of functions containing
$\mathcal{F}\cup\left\{\id_{A_\xi}\right\}_{\xi\in I}$ and is closed under orderly compositions.
\end{dfn}

The collection of orderly terms over $\mathcal{F}$ is in fact the collection of operations on $(A_\xi)_{\xi\in I}$ which can be generated by an
application of finitely many of the following rules:
\begin{enumerate}
\item for each $\xi\in I$, the identity function $\text{id}_{A_\xi}$ is an orderly term,
\item every operation  in $\mathcal{F}$ is an orderly term,
\item if $f$ is an operation on $(A_\xi)_{\xi\in I}$ given by $f(\bar{x}_1, \ldots, \bar{x}_k)=g(h_1(\bar{x}_1)$, $\ldots$, $h_k(\bar{x}_k))$ for some
    $g\in\mathcal{F}$ and some orderly terms $h_1, \ldots, h_k$, then $f$ is an orderly term.
\end{enumerate}

We denote the concatenation operation by $\ast$ in this paper.

\begin{dfn}[Reduction $\leq_\mathcal{F}$]\label{reduction}
Let $((A_\xi)_{\xi\in I}, \mathcal{F})$ be an algebra and let $\vec{a}$ and $\vec{b}$ be members of ${^\omega}\left(\bigcup_{\xi\in I}A_\xi\right)$.
Then $\vec{a}$ is said to be a \emph{reduction} of $\vec{b}$ if there exist orderly terms $f_j$ over $\mathcal{F}$ and finite subsequences $\vec{b}_j$
of $\vec{b}$ such that
\begin{enumerate}
\item $\vec{a}(j)=f_j(\bar{b}_j)$ for each $j\in\omega$ and
\item $\vec{b}_0\ast\vec{b}_1\ast\cdots$ forms a subsequence of $\vec{b}$.
\end{enumerate}
We write $\vec{a}\leq_\mathcal{F}\vec{b}$ to mean $\vec{a}$ is a reduction of $\vec{b}$.
\end{dfn}

The relation $\leq_\mathcal{F}$ is a preorder\footnote{A preorder is a relation that is reflexive and transitive.} on ${^\omega\left(\bigcup_{\xi\in
I}A_\xi\right)}$. Note that the inclusion of the identity functions in the set of orderly terms is necessary to ensure that the relation
$\leq_\mathcal{F}$ is reflexive.

\begin{dfn}\label{sort}
Suppose $((A_\xi)_{\xi\in I}, \mathcal{F})$ is an algebra and $\vec{e}\in{^\omega}I$. We say that $\vec{a}\in{^\omega\left(\bigcup_{\xi\in
I}A_\xi\right)}$ is \emph{$\vec{e}$-sorted} if $\vec{a}(n)\in A_{\vec{e}(n)}$ for each $n\in\omega$ and that $\vec{e}$ is the \emph{sort} of $\vec{a}$.
\end{dfn}

Note that the sort of a sequence is unique under our assumption that the sets in the family of phyla $(A_\xi)_{\xi\in I}$ are pairwise disjoint.

\begin{dfn}\label{FR}
If $\vec{b}$ is an $\vec{e}$-sorted sequence, define
\begin{equation*}
\FR_\mathcal{F}^{\vec{e}}(\vec{b})=\left\{\vec{a}(0):\vec{a}\leq_\mathcal{F}\vec{b}\;\text{and}\;\vec{a}\;\text{is}\;\vec{e}\text{-sorted}\right\}.
\end{equation*}
\end{dfn}

We end this section with the central notion of the paper.

\begin{dfn}[$\vec{e}$-Ramsey Algebra]\label{RA}
Let $\vec{e}$ be a sort. An algebra $((A_\xi)_{\xi\in I}, \mathcal{F})$ is said to be an $\vec{e}$-Ramsey algebra if, for each $\vec{e}$-sorted sequence
$\vec{b}$ and each $X\subseteq A_{\vec{e}(0)}$, there exists an $\vec{e}$-sorted reduction $\vec{a}$ such that $\FR^{\vec{e}}_\mathcal{F}(\vec{a})$ is
either contained in or disjoint from $X$.

Such a sequence $\vec{a}$ is said to be \emph{homogeneous} for $X$ (with respect to $\mathcal{F}$).
\end{dfn}

In the case where the algebra is homogeneous, the set of sorts would be a singleton and we shall drop the reference to $\vec{e}$ and refer any such
$\vec{e}$-Ramsey algebra as a \emph{Ramsey algebra}. Also, the set defined by Eq. \ref{FR} can be characterized by
\begin{equation}\label{homFR}
\fr_\mathcal{F}^{\vec{e}}(\vec{b})=\{f(\bar{b}_0):f\in\ot(\mathcal{F})\;\text{and}\;\vec{b}_0\;\text{is a finite subsequence of}\;\vec{b}\},
\end{equation}
in which case the notation will be simplified to $\fr_\mathcal{F}(\vec{b})$.

\begin{rem}\label{finitepartition}
If $\mathcal{A}$ is an $\vec{e}$-Ramsey algebra for a given $\vec{e}$, then for each finite partition $Q_1\cup\cdots\cup Q_N$ of $A_{\vec{e}(0)}$ and
for each $\vec{e}$-sorted $\vec{b}$, there exists an $\vec{e}$-sorted $\vec{a}\leq_\mathcal{F}\vec{b}$ and an $i^\ast\in\{1, \ldots, N\}$ such that
$\fr_\mathcal{F}^{\vec{e}}(\vec{a})\subseteq Q_{i^\ast}$.
\end{rem}

A classic result of Ramsey algebra is that of Hindman:
\begin{thm}[Hindman]\label{semigroup}
Every semigroup is a Ramsey algebra.
\end{thm}

\section{From Ramsey Spaces to Ramsey Algebras}\label{space}
The origin of Ramsey algebras has its roots in the notion a Ramsey space introduced by Carlson. As mentioned in the introductory section, largely due to
the ensuing work of Todorcevic \cite{sT10}, this notion is now referred to as a topological Ramsey space. This section is intended to give a short
account of the connection between Ramsey algebras and topological Ramsey spaces.

If $R$ is a nonempty set of infinite sequences equipped with a preorder $\preceq$, then for all $\vec{b}\in R$ and all $n\in\omega$, the sets
\begin{equation}\label{nbhdbasis}
[n, \vec{b}]=\{\vec{a}\in R:\vec{a}\preceq\vec{b}\;\text{and}\;\vec{a}\upharpoonright n=\vec{b}\upharpoonright n\}
\end{equation}
form a neighborhood basis of the \emph{natural topology} on $R$.

\begin{dfn}
Let $R$ be a nonempty set of infinite sequences equipped with a preorder $\preceq$ and let $X\subseteq R$. The set $X$ is said to be \emph{Ramsey} if,
for each $\vec{b}\in R$ and each $n\in\omega$, there exists $\vec{a}\in[n, \vec{b}]$ such that $[n, \vec{a}]\subseteq X$ or $[n, \vec{a}]\subseteq X^C$.
The set $X$ is said to be \emph{Ramsey null} in the event there exists an $\vec{a}\in[n, \vec{b}]$ for every given $\vec{b}$ and $n\in\omega$ such that
$[n, \vec{a}]\subseteq X^C$.
\end{dfn}

\begin{dfn}\label{RSdfn}
If $R$ is a nonempty set of infinite sequences equipped with the preorder $\preceq$, then $(R, \preceq)$ endowed with the natural topology is called a
\emph{topological Ramsey space} if every set which has the Property of Baire is Ramsey and every meager set is Ramsey null\footnote{Under the Axiom of
Choice, the latter property can be discarded.}.
\end{dfn}

The neighborhood sets given in Eq. \ref{nbhdbasis} and the natural topology so generated are defined with the preorder induced by the reduction
operation within an algebra $((A_\xi)_{\xi\in I}, \mathcal{F})$ in mind. We recall from Section 2 that the relation $\leq_\mathcal{F}$ for any given
algebra $((A_\xi)_{\xi\in I}, \mathcal{F})$ is a preordering on the infinite sequences of ${^\omega\!\left(\bigcup_{\xi\in I}A_\xi\right)}$; we will be
chiefly interested in the natural topology on subsets of ${^\omega\!\left(\bigcup_{\xi\in I}A_\xi\right)}$ determined by any given sort $\vec{e}$ as
these are the topologies connected to the notion of a $\vec{e}$-Ramsey algebra. For each sort $\vec{e}$, define $R^{\vec{e}}$ to be the set consisting
of $\vec{e}$-sorted sequences of $\bigcup_{\xi\in I}A_\xi$ equipped with the relation $\leq_\mathcal{F}\upharpoonright (R^{\vec{e}}\times R^{\vec{e}})$.
We will abuse notation and denote such restriction by the same symbol $\leq_\mathcal{F}$ when no confusion arises. Given an algebra $((A_\xi)_{\xi\in
I}, \mathcal{F})$ and a sort $\vec{e}$, we write $\mathfrak{R}^{\vec{e}}((A_\xi)_{\xi\in I}, \mathcal{F})$ for $\left(R^{\vec{e}},
\leq_\mathcal{F}\right)$ to make explicit the algebra involved. For emphasis, we also add $\vec{e}$ as superscript to the members of the neighborhood
basis $[n, \vec{b}]^{\vec{e}}=\{\vec{a}\in R^{\vec{e}}:\vec{a}\leq_\mathcal{F}\vec{b}\;\text{and}\;\vec{a}\upharpoonright n=\vec{b}\upharpoonright n\}$
to indicate the sort of the sequences in question.

The next two theorems of Carlson \cite{carlson1988some}, adapted to the current context, are key to relating the notion of a Ramsey algebra to that of a
topological Ramsey space. The next theorem can be thought of as an abstract Ellentuck Theorem.

\begin{thm}\label{CarlsonRamseySpace}
Suppose $\left((A_\xi)_{\xi\in I}, \mathcal{F}\right)$ is an algebra, where $\mathcal{F}$ is a finite family of \emph{nonunary} operations. For each
sort $\vec{e}$, $\mathfrak{R}^{\vec{e}}\left((A_\xi)_{\xi\in I}, \mathcal{F}\right)$ is a topological Ramsey space if and only if
$\mathfrak{R}^{\vec{e}}\left((A_\xi)_{\xi\in I}, \mathcal{F}\right)$ satisfies the following for each $n\in\omega$:
\begin{enumerate}
\item[\emph{($\clubsuit$)}] whenever $\vec{b}\in R^{\vec{e}}$ and $X$ is a set of initial segments of length $n+1$ of $\vec{e}$-sorted sequences ,
    then there exists $\vec{a}\in[n, \vec{b}]^{\vec{e}}$ such that the set consisting of all the initial segments of sequences in $[n,
    \vec{a}]^{\vec{e}}$ of length $n+1$ is either a subset of $X$ or is disjoint from $X$.
\end{enumerate}
\end{thm}

\begin{thm}\label{CarlsonMain}
Under the hypothesis of Theorem \ref{CarlsonRamseySpace}, $\mathfrak{R}^{\vec{e}}\left((A_\xi)_{\xi\in I}, \mathcal{F}\right)$ is a topological Ramsey
space if and only if, for each $m\in\omega$, $\mathfrak{R}^{\vec{e}-m}\left((A_\xi)_{\xi\in I}, \mathcal{F}\right)$ satisfies the special case of
$\clubsuit$ with $n=0$.
\end{thm}

The preceding theorem is the key to relating topological Ramsey spaces to Ramsey algebra. When paraphrased, it becomes a statement in the language of
Ramsey algebra.

\begin{cor}\label{concuse}
Under the hypothesis of Theorem \ref{CarlsonRamseySpace}, $\mathfrak{R}^{\vec{e}}\left((A_\xi)_{\xi\in I}, \mathcal{F}\right)$ is a topological Ramsey
space if and only if $((A_\xi)_{\xi\in I}, \mathcal{F})$ is an $(\vec{e}-m)$-Ramsey algebra for each $m\in\omega$.

In particular, if $((A_\xi)_{\xi\in I}, \mathcal{F})$ is homogeneous, then $\mathfrak{R}^{\vec{e}}\left((A_\xi)_{\xi\in I}, \mathcal{F}\right)$ is a
topological Ramsey space if and only if $((A_\xi)_{\xi\in I}, \mathcal{F})$ is a Ramsey algebra.
\end{cor}

\section{Some Elementary Results}
Given any algebra $\mathcal{A}$ with the indexing set $I$, we single out the following class of sorts:
\begin{equation*}
\Omega=\{\vec{e}\in{^\omega}I:\text{if}\;\vec{e}(i)=\xi\;\text{for some}\;i,\;\text{then}\;|\{i:\vec{e}(i)=\xi\}|=\aleph_0\}.
\end{equation*}

We will mainly be studying $\vec{e}$-Ramsey algebras for $\vec{e}$ belonging in this class. The reason for this choice is that sequences of sorts
belonging in $\Omega$ bear resemblance to properties familiar from homogeneous Ramsey algebras. For example, if $\vec{e}\in\Omega$, then
\begin{equation}\label{Omegafr}
c\in\fr_\mathcal{F}^{\vec{e}}(\vec{b})\Longleftrightarrow c=f(\tau)
\end{equation}
for some $f\in\ot(\mathcal{F})$ and some finite subsequence $\tau$ of $\vec{b}$. Compare this to Eq. \ref{homFR}. The reader can go through the
definition of a reduction to verify Eqv. \ref{Omegafr}.

As the first term of a sort is of critical importance in reductions, we break the set $\Omega$ down as follows: For each $\xi\in I$,
\begin{equation*}
\Omega_\xi=\{\vec{e}\in\Omega:\vec{e}(0)=\xi\}.
\end{equation*}

\begin{lem}\label{frsubsets}
Suppose that $\vec{e}\in\Omega$, $\vec{e}'$ is any sort such that $\vec{e}\:'(0)=\vec{e}(0)$, $\vec{a}$ is $\vec{e}$-sorted, $\vec{a}\:'$ is
$\vec{e}\:'$-sorted, and $\vec{a}\:'\leq_\mathcal{F}\vec{a}$. Then,
$\fr_\mathcal{F}^{\vec{e}\:'}(\vec{a}\:')\subseteq\fr_\mathcal{F}^{\vec{e}}(\vec{a})$.
\end{lem}
\begin{proof}
Let $\vec{c}\:'\leq_\mathcal{F}\vec{a}\:'$ be $\vec{e}\:'$-sorted; we want to show that $\vec{c}\:'(0)$ is also a member of
$\fr_\mathcal{F}^{\vec{e}}(\vec{a})$. To do this, note that $\vec{c}\;'\leq_\mathcal{F}\vec{a}$ by the transitivity of $\leq_\mathcal{F}$. Now, apply
Eq. \ref{Omegafr} and we see that $\vec{c}\:'(0)$ is indeed a member of $\fr_\mathcal{F}^{\vec{e}}(\vec{a})$.
\end{proof}

A classification of homogeneous algebras all of whose operations are unary can be found in \cite{teh2016ramsey}. We state it here for reference
purposes:

\begin{thm}\label{unaryS}
If $\mathcal{F}$ is a collection of unary operations, then $(A, \mathcal{F})$ is a Ramsey algebra if and only if every element of $A$ can be sent into
an element of $S=\{a\in A:f(a)=a\;\text{for all}\;f\in\mathcal{F}\}$ by finitely many applications of the members of $\mathcal{F}$.
\end{thm}

Theorem \ref{unaryS}'s heterogeneous analogue is given in the next theorem, but before then we will content ourselves with two lemmas. In Lemma
\ref{disjointops} and Theorem \ref{premprop} below, we consider algebras whose family $\mathcal{F}$ of operations are disjoint in the following sense:
$\mathcal{F}=\bigcup_{\xi\in I}\mathcal{G}_\xi$, where $\mathcal{G}_\xi=\{f\in\mathcal{F}:\dom(f)=A_\xi^n\;\text{for
some}\;n\in\omega\;\text{and}\;\cod(f)=A_\xi\}$ for each $\xi\in I$.

\begin{lem}\label{disjointops}
Suppose $\mathcal{A}=\left(\{A_\xi\}_{\xi\in I}, \mathcal{F}\right)$ is an algebra such that $\mathcal{F}=\bigcup_{\xi\in I}\mathcal{G}_\xi$. Then:
\begin{enumerate}
\item $\ot(\mathcal{F})=\bigcup_{\xi\in I}\ot\left(\mathcal{G}_\xi\right)$ and this is a \emph{disjoint} union.
\item For each $\vec{e}\in\Omega_\eta$, if $\vec{b}$ is $\vec{e}$-sorted and $\vec{\beta}$ is the subsequence of $\vec{b}$ consisting of members of
    $A_\eta$, then $\fr_\mathcal{F}^{\vec{e}}(\vec{b})=\fr_{\mathcal{G}_\eta}^{\vec{e}}(\vec{b})=\fr_{\mathcal{G}_\eta}(\vec{\beta})$.
\item If $\vec{e}\not\in\Omega$ and $\vec{b}$ is $\vec{e}$-sorted, then each $\vec{e}$-sorted $\vec{a}\leq_\mathcal{F}\vec{b}$ has the property that
    $\vec{a}(0)$ is the image of $\vec{b}(0)$ under a composition of some unary operations in $\mathcal{G}_{\vec{e}(0)}$. In particular, if
    $\mathcal{G}_{\vec{e}(0)}$ does not contain unary operations on $A_{\vec{e}(0)}$, then
    $\fr_\mathcal{F}^{\vec{e}}(\vec{b})=\left\{\vec{b}(0)\right\}$.
\end{enumerate}
\end{lem}
\begin{proof} (1) \& (2) Apart from tedious definition chasing and an appeal to Eqv. \ref{Omegafr}, these two parts of the lemma are intuitive.

(3) Let $\eta\in I$ be an index that occurs only finitely often in $\vec{e}$ and let $N\in\omega$ be such that $\vec{e}(N)=\eta$ and
$\vec{e}(i)\neq\eta$ for all $i>N$. Then, since $\mathcal{F}$ is a disjoint union of the
$\mathcal{G}_\xi$'s, we see from Part 1 that, if $\vec{a}\leq_\mathcal{F}\vec{b}$, then $\vec{a}(N)$ must be the image of $\vec{b}(N)$
under a unary $f\in\ot(\mathcal{G}_\eta)$, i.e. $\vec{a}(N)$ is the image
of a composition of unary operations of $\mathcal{G}_\eta$. In fact, $\vec{a}(i)$ must be an image of
$\vec{b}(i)$ under a unary orderly term of $\mathcal{G}_{\vec{e}(i)}$ for each $i\in\{0, \ldots, N\}$. In particular, if
$\mathcal{G}_{\vec{e}(0)}$ does not contain unary operations on $A_{\vec{e}(0)}$, then $\vec{a}(0)=\vec{b}(0)$ since the only
unary operation in $\mathcal{G}_{\vec{e}(0)}$ is $\id_{A_{\vec{e}(0)}}$; therefore, $\fr_\mathcal{F}^{\vec{e}}(\vec{b})=\left\{\vec{b}(0)\right\}$.
\end{proof}

If $T$ is an operation on a set $A$ and $X\subseteq A$, then $T[X]$ will denote the image of $X$ under $T$, i.e. $T[X]=\{T(x):x\in X\}$.

\begin{lem}[Katet\v{o}v \cite{katetov67}]\label{katetov}
Let $A$ be a set and suppose that $T:A\to A$ does not have fixed points. Then, there exists a partition $A=P_1\cup P_2\cup P_3$ such that $T[P_i]\cap
P_i=\varnothing$ for each $i=1, 2, 3$.
\end{lem}

\begin{thm}\label{premprop}
Suppose $\mathcal{A}=\left(\{A_\xi\}_{\xi\in I}, \mathcal{F}\right)$ is an algebra such that $\mathcal{F}=\bigcup_{\xi\in I}\mathcal{G}_\xi$. Then:
\begin{enumerate}
\item For each $\vec{e}\in\Omega_\eta$, $\mathcal{A}$ is an $\vec{e}$-Ramsey algebra if and only if $(A_\eta, \mathcal{G}_\eta)$ is a Ramsey algebra.
\item Suppose $\vec{e}\not\in\Omega$. Then, $\mathcal{A}$ is an $\vec{e}$-Ramsey algebra if and only if every element of $A_{\vec{e}(0)}$ can be sent
    into the set of fixed points $S_{\vec{e}(0)}=\left\{c\in A_{\vec{e}(0)}:f(c)=c\;\text{for each unary}\;f\in\mathcal{G}_{\vec{e}(0)}\right\}$ by
    finitely many applications of the unary operations $f\in\mathcal{G}_{\vec{e}(0)}$. In particular, if $\mathcal{G}_{\vec{e}(0)}$ does not have
    unary operations, then $S_{\vec{e}(0)}=A_{\vec{e}(0)}$ and so $\mathcal{A}$ is an $\vec{e}$-Ramsey algebra.
\end{enumerate}
\end{thm}
\begin{proof}
(1)($\Rightarrow$) Suppose that $\mathcal{A}$ is an $\vec{e}$-Ramsey algebra for $\vec{e}\in\Omega_\eta$ and let $\vec{\beta}$ be an infinite sequence
of $A_\eta$ and $X\subseteq A_\eta$. Now pick an $\vec{e}$-sorted sequence $\vec{b}$ such that its subsequence all of whose terms are members of
$A_\eta$ is $\vec{\beta}$.

Since $\mathcal{A}$ is an $\vec{e}$-Ramsey algebra, let $\vec{a}\leq_\mathcal{F}\vec{b}$ be $\vec{e}$-sorted and homogeneous for $X$. Thus, define the
infinite sequence $\vec{\alpha}$ of $A_\eta$ as the subsequence of $\vec{a}$ all of whose terms lie in $A_\eta$. Observe that
$\vec{\alpha}\leq_{\mathcal{G}_\eta}\vec{\beta}$, a property again guaranteed by the disjoint composition of $\mathcal{F}$. We may now apply Part 2 of
Lemma \ref{disjointops} to obtain $\fr_{\mathcal{G}_\eta}(\vec{\alpha})=\fr_\mathcal{F}^{\vec{e}}(\vec{a})$. As a result, the homogeneity of
$\vec{\alpha}$ for $X$ entails.

($\Leftarrow$) Suppose that $(A_\eta, \mathcal{G}_\eta)$ is a Ramsey algebra and let $\vec{b}$ be an $\vec{e}$-sorted sequence and $X\subseteq A_\eta$.
To see that $\mathcal{A}$ is an $\vec{e}$-Ramsey algebra, it suffices to show that, if $\vec{\beta}$ is the subsequence of $\vec{b}$ consisting of terms
lying in $A_\eta$ and $\vec{\alpha}\leq_{\mathcal{G}_\eta}\vec{\beta}$ is homogeneous for $X$, then there exists an $\vec{e}$-sorted
$\vec{a}\leq_\mathcal{F}\vec{b}$ such that $\vec{a}$ is homogeneous for $X$.

Indeed, if $\vec{\beta}, \vec{\alpha}$ are as described, then for each $i\in\omega$, let $g_i\in\ot(\mathcal{G}_\eta)\subseteq\ot(\mathcal{F})$ and let
$\tau_i$ be a subsequence of $\vec{\beta}$ such that $g_i(\tau_i)=\vec{\alpha}(i)$ and $\tau_0\ast\tau_1\ast\cdots$ forms a subsequence of
$\vec{\beta}$. Since $\vec{\beta}$ is a subsequence of $\vec{b}$, it follows that $\tau_0\ast\tau_1\ast\cdots$ is also a subsequence of $\vec{b}$.

We then construct the desired $\vec{a}$ recursively as follows:

First, set $\vec{a}(0)=\vec{\alpha}(0)$.

Next, suppose that $\vec{a}$ has been constructed up to $\vec{a}(N)$ in such a way that
\begin{enumerate}
\item $\vec{a}(0)=f_0(\sigma_0), \ldots, \vec{a}(N)=f_N(\sigma_N)$ for some subsequences $\sigma_0, \ldots, \sigma_N$ of $\vec{b}$ and some $f_0,
    \ldots, f_N\in\ot(\mathcal{F})$, and $\sigma=\sigma_0\ast\cdots\ast\sigma_N$ forms a subsequence of $\vec{b}$,
\item the subsequence of $\langle\vec{a}(0), \ldots, \vec{a}(N)\rangle$ whose terms are members of $A_\eta$ is a subsequence of $\vec{\alpha}$.
\end{enumerate}

We then decide the value of $\vec{a}(N+1)$ as follows:
\begin{enumerate}
\item[(i)] If $\vec{e}(N+1)=\eta$, then we choose an $i^\ast$ so that $\tau_{i^\ast}$ is a subsequence of the tail of $\vec{b}$ without $\sigma$. We
    then set $\vec{a}(N+1)=\vec{\alpha}(i^\ast)=g_{i^\ast}(\tau_{i^\ast})$.
\item[(ii)] If $\vec{e}(N+1)=\xi\neq\eta$, then we pick the first term belonging in $A_\xi$ in the tail of $\vec{b}$ without $\sigma$ and let it be
    the value of $\vec{a}(N+1)$.
\end{enumerate}

The sequence $\vec{a}$ so constructed is $\vec{e}$-sorted and is such that $\vec{a}\leq_\mathcal{F}\vec{b}$. In addition, its subsequence
$\vec{\alpha}'$ of terms that belong in $A_\eta$ forms a subsequence of $\vec{\alpha}$. As such, we have
$\fr_{\mathcal{G}_\eta}(\vec{\alpha}')\subseteq\fr_{\mathcal{G}_\eta}(\vec{\alpha})$. Now, by Part 2 of Lemma \ref{disjointops}, we have
$\fr_\mathcal{F}^{\vec{e}}(\vec{a})=\fr_{\mathcal{G}_\eta}(\vec{\alpha}')$. The homogeneity of $\vec{a}$ for $X$ then follows from the homogeneity of
$\vec{\alpha}$ for $X$.

(2)($\Leftarrow$) Let $\vec{b}$ be an $\vec{e}$-sorted sequence of let $X\subseteq A_{\vec{e}}$ be given. Using Lemma \ref{disjointops}, pick an
$\vec{e}$-sorted $\vec{a}\leq_\mathcal{F}\vec{b}$ such that  $\vec{a}(0)\in S_{\vec{e}(0)}$. Then, it follows from the definition of $S_{\vec{e}(0)}$
and Lemma \ref{disjointops} that $\fr_\mathcal{F}^{\vec{e}}(\vec{a})=\{\vec{a}(0)\}$, therefore, $\vec{a}$ is homogeneous for $X$.

($\Rightarrow$) The proof of this direction can be traced back to the proof of Theorem 4.2 of \cite{teh2016ramsey}. Let $\vec{e}\not\in\Omega$ be given
and let $\vec{e}(0)=\eta$. Mimicking the proof there, let $\alpha$ be a symbol not already in $A_\eta$ and fix an $a_0\in A_\eta$. We define an
operation $T$ on $A_\eta\cup\{\alpha\}$ such that
\begin{enumerate}
\item[(a)] if $a\in S_\eta$, then $T(a)=\alpha$,
\item[(b)] $T(\alpha)=a_0$, and
\item[(c)] if $a\in A_\eta\setminus S_\eta$, then $T(a)=f(a)\neq a$ for some $f\in\mathcal{G}_\eta$.
\end{enumerate}

Such $T$ does not have fixed points and so, by Lemma \ref{katetov}, let $P_1\cup P_2\cup P_3$ be a partition of $A_\eta\cup\{\alpha\}$ such that
$T[P_i]\cap P_i=\varnothing$ for each $i=1, 2, 3$. This induces a partition $Q_1\cup Q_2\cup Q_3$ of $A_\eta$ such that whenever $a\in A\cap S_\eta^C$,
if $a\in Q_i$, then $f(a)\not\in Q_i$ for some $f\in\mathcal{G}_\eta$.

Now, suppose that $\mathcal{A}$ is an $\vec{e}$-Ramsey algebra. Towards a contradiction, assume that there exists $c\in A_\eta$ such that $c$ cannot be
sent into $S_\eta$ by finitely many applications of the unary operations of $\mathcal{G}_\eta$. Fix one such element $c$ and pick any $\vec{e}$-sorted
$\vec{b}$ such that $\vec{b}(0)=c$. Referencing Remark \ref{finitepartition}, there exists an $\vec{e}$-sorted $\vec{a}\leq_\mathcal{F}\vec{b}$ and an
$i^\ast\in\{1, 2, 3\}$ such that $\fr_\mathcal{F}^{\vec{e}}(\vec{a})\subseteq Q_{i^\ast}$. By Part 3 of Lemma \ref{disjointops}, $\vec{a}(0)$ is
obtained from $\vec{b}(0)=c$ by finitely many applications of the unary operations of $\mathcal{G}_\eta$, therefore $\vec{a}(0)\not\in S_\eta$ by our
choice of $c$. The operation $T$ now furnishes an $f\in\mathcal{G}_\eta$ such that $f(\vec{a}(0))\not\in Q_{i^\ast}$, which contradicts
$f(\vec{a}(0))\in\fr_\mathcal{F}(\vec{a})\subseteq Q_{i^\ast}$.
\end{proof}

A simple example will show us that, if heterogeneous operations are present, then the conclusion for Part 1 above need not hold:

\begin{eg}\label{egconstantbinary}
Suppose $\mathcal{A}=(A_0, A_1, \circ, h)$, where $h:A_1\to A_0$ is constant with value $a_1$ and $\circ:A_0^2\to A_0$ is constant with value $a_0$, for
which $a_0\neq a_1$. Let $\alpha_0\in A_1$ and let $\vec{e}\in\Omega_0$ be nonconstant. Note that $(A_0, \circ)$ is a Ramsey algebra. However, for each
$\vec{e}$-sorted $\vec{a}\leq_\mathcal{F}\vec{b}$, the set $\fr_\mathcal{F}^{\vec{e}}(\vec{a})$ has at least two elements $a_0, a_1$ and so no such
$\vec{c}$ can be homogeneous for $\{a_0\}\subseteq A_0$.
\end{eg}

Now, let us call by $\Omega^J$ the set of those $\vec{e}\in\Omega$ such that all indices $\xi\in I$ appearing in $\vec{e}$ form the subset $J\subseteq
I$. Also, let $\Omega_\eta^J=\left\{\vec{e}\in\Omega^J:\vec{e}\in\Omega_\eta\right\}$.

\begin{thm}\label{anysort}
For any family $\mathcal{F}$ of operations, if $\mathcal{A}=\left(\bigcup_{\xi\in I}A_\xi, \mathcal{F}\right)$ is an $\vec{e}$-Ramsey algebra for an
$\vec{e}\in\Omega_\eta^J$, then it is an $\vec{e}$-Ramsey algebra for all $\vec{e}\in\Omega_\eta^J$.
\end{thm}
\begin{proof}
Let $\vec{e}, \vec{e}\:'\in\Omega_\eta^J$ and suppose that $\mathcal{A}$ is an $\vec{e}$-Ramsey algebra. It suffices to prove that if $\mathcal{A}$ is
an $\vec{e}$-Ramsey algebra, then it is an $\vec{e}\:'$-Ramsey algebra.

Let the $\vec{e}\:'$-sorted sequence $\vec{b}'$ and $X\subseteq A_\eta$ be given. We obtain an $\vec{e}$-sorted subsequence $\vec{b}$ of $\vec{b}'$,
which is possible since $\vec{e}\:'\in\Omega^J$. Then, choose an $\vec{e}$-sorted $\vec{a}\leq_\mathcal{F}\vec{b}$ such that $\vec{a}$ is homogeneous
for $X$ as we have assumed that $\mathcal{A}$ is an $\vec{e}$-Ramsey algebra. Next, we obtain an $\vec{e}\:'$-sorted subsequence $\vec{a}'$ of
$\vec{a}$, possible since $\vec{e}\in\Omega^J$.

By the transitivity of $\leq_\mathcal{F}$, we have $\vec{a}'\leq_\mathcal{F}\vec{b}'$; additionally,
$\fr_\mathcal{F}^{\vec{e}\:'}(\vec{a}')\subseteq\fr_\mathcal{F}^{\vec{e}}(\vec{a})$ following Lemma \ref{frsubsets}, whence the homogeneity of
$\fr_\mathcal{F}^{\vec{e}\:'}(\vec{a}')$ for $X$ entails from the homogeneity of $\fr_\mathcal{F}^{\vec{e}}(\vec{a})$ for $X$.
\end{proof}

We want to emphasize that the preceding theorem holds because, as we see in its proof, the sorts $\vec{e}$ and $\vec{e}\:'$ share the same first term
and the same set $J$ of indices. Part 1 of Theorem \ref{premprop} illustrates the dependence on the first term, whereas one can easily cook up an
example to demonstrate that sharing the same $J$ is indeed essential.

\section{Classification of Ramsey Vector Spaces}\label{main}
We now begin the study of vector spaces in the context of Ramsey algebra. For completeness sake, we carry out an analysis for all possible sorts in this
study. A vector space is a structure $(\mathbb{V}, \mathbb{F}, +_\mathbb{F}, \times_\mathbb{F}, +_\mathbb{V}, \cdot)$, where $(\mathbb{V},
+_\mathbb{V})$ forms an abelian group, $(\mathbb{F}, +_\mathbb{F}, \times_\mathbb{F})$ a field whose elements are called scalars, and
$\cdot:\mathbb{F}\times \mathbb{V}\rightarrow \mathbb{V}$ is scalar multiplication such that $1\cdot v=v$ for all $v\in\mathbb{V}$, and the following
axioms of distributivity hold:
\begin{enumerate}
\item $(r+_\mathbb{F} s)\cdot v=r\cdot v+_\mathbb{F} s\cdot v$ for each $r, s\in \mathbb{F}, v\in \mathbb{V}$;
\item $r\cdot(u+_\mathbb{V} v)=r\cdot u+_\mathbb{V} r\cdot v$ for each $r\in \mathbb{F}, u, v\in \mathbb{V}$.
\end{enumerate} Vector spaces are two-sorted heterogeneous structure: $(\mathbb{V}, \mathbb{F}, +_\mathbb{V}, +_\mathbb{F}, \times_\mathbb{F},
\cdot)$. In this section, $\mathcal{F}$ is reserved for the family $\{+_\mathbb{V}, +_\mathbb{F}, \times_\mathbb{F}, \cdot\}$ and we always let $A_0$
denote the underlying field of a vector space and $A_1$ the set of vectors. The uppercase ``oh'' $O$ is reserved for the zero vector in any vector
space under discussion.

A few observations regarding the orderly terms of $\mathcal{F}$ will be helpful.

\begin{lem}\label{typesoff}
Let $(\mathbb{V}, \mathbb{F}, \mathcal{F})$ be a vector space and $f\in\emph{OT}(\mathcal{F})$. The following hold:
\begin{enumerate}
\item The only unary orderly terms over $\mathcal{F}$ are the identity functions $\id_\mathbb{F}$ and $\id_\mathbb{V}$; the only binary orderly terms
    over $\mathcal{F}$ are the binary operations in $\mathcal{F}$.
\item If $\dom(f)=\mathbb{F}^n$ for some nonzero $n\in\omega$, then $\rn(f)\subseteq \mathbb{F}$.
\item If $\dom(f)=A_{\xi_1}\times\cdots\times A_{\xi_n}$ for some natural number $n\geq 1$ and there exist $i\leq n$ such that $\xi_i=1$, then it
    follows that $\rn(f)\subseteq \mathbb{V}$.
\item If $f$ satisfies the hypothesis of Fact 3, then $\xi_n=1$ (i.e. $A_{\xi_n}=\mathbb{V}$).
\item If f is vector valued and all the scalar components of $\bar{x}$ have the value $0$, then $f(\bar{x})$ is either the zero vector or the sum of
    some vectors appearing in $\bar{x}$; if $f$ is scalar valued, then $f(\bar{x})=0$.
\end{enumerate}
\end{lem}
\begin{proof}
Carrying out an induction argument on the generation of the orderly terms will establish Facts 1 and 2.

Facts 3 through 5 can also be established by induction on the generation of orderly terms. We will only give the proof of Fact 4 and a sketch of the
proof of Fact 5.

Fact 4: The implication holds if $f$ is either $\text{id}_\mathbb{F}$ or $\text{id}_\mathbb{V}$. Similarly, if $f\in\mathcal{F}$, then the antecedent of
the implication holds provided $f$ is either scalar multiplication or vector addition. In both cases, the consequent of the implication holds.

Now, suppose $f(\bar{x}_1, \bar{x}_2)=g(h_1(\bar{x}_1), h_2(\bar{x}_2))$ for some $g\in\mathcal{F}$ and orderly terms $h_1$ and $h_2$, where the arity
of $g$ owes to the fact that every $g\in\mathcal{F}$ is binary. If $g$ is field addition or field multiplication, then both $h_1$ and $h_2$ are
scalar-valued and hence $\bar{x}_1$ and $\bar{x}_2$ are all scalar terms, contradicting $\xi_n=1$. Thus, $g$ is either vector addition or scalar
multiplication. In both cases, $h_2$ is vector valued. Hence, $\bar{x}_2$ involves some vector term. Then, by the induction hypothesis, the trailing
term of $\bar{x}_2$ is a vector.

Fact 5: For the base cases, the conclusion is obvious. Thus, suppose $f(\bar{x})=g(h_1(\bar{x}_1)$, $h_2(\bar{x}_2))$ with $g\in\mathcal{F}$ and with
$h_1, h_2\in\text{OT}(\mathcal{F})$ both satisfying the conclusion of Fact 5. Now assume $\bar{x}$ satisfies the hypothesis of Fact 5. If $g$ is vector
addition, then $f(\bar{x})=h_1(\bar{x}_1)+h_2(\bar{x}_2)$, and so $f$ clearly satisfies the conclusion of Fact 5; this is the case where the result is
the nonempty sum of some vectors from $\bar{x}$. Proofs of the other cases of $g$ are similarly obvious and will, hence, be omitted.
\end{proof}

The next lemma will be helpful in proving various Ramsey-algebraic properties of vector spaces.

\begin{lem}\label{sortietype}
Let $(\mathbb{V}, \mathbb{F}, \mathcal{F})$ be a vector space, let $\vec{b}$ be an $\vec{e}$-sorted sequence of elements of $\mathbb{F}\cup \mathbb{V}$,
and let $\vec{a}$ be an $\vec{e}$-sorted reduction of $\vec{b}$. If $\vec{e}\not\in\Omega$ with $n^*\in\omega$ being \emph{least} such that
$\vec{e}-(n^*+1)$ is constant, then $\vec{a}(i)=\vec{b}(i)$ for each $i\leq n^*$.
\end{lem}
\begin{proof} For any such sort $\vec{e}$, the eventually constant value can be either $0$ or $1$. The proofs of both cases are similar and we will only
provide for the case that is eventually constant with value $0$. In such a case, $\vec{e}(n^*)=1$ and $\vec{e}(i)=0$ for all $i>n^*$.

To begin the proof, for each $i\leq n^*$, let $\vec{b}_i$ and $f_i\in\text{OT}(\mathcal{F})$ be such that $\vec{a}(i)=f_i(\bar{b}_i)$ and
$\vec{b}_1\ast\cdots\ast\vec{b}_{n^*}$ is a subsequence of $\vec{b}$. Let the leading term of $\vec{b}_{n^*}$ be $\vec{b}(j)$. We claim that $n^*=j$.
Clearly, $n^*\leq j$. Now assume that $n^*<j$. Then, we would have $f_{n^*}$ operating on scalars, yielding a scalar (Fact 2, Lemma \ref{typesoff}),
contradicting $\vec{a}(n^*)$ being a vector. That $f_{n^*}$ is unary follows from Fact 4 of Lemma \ref{typesoff}.

It now follows by the pigeonhole principle that $f_i$ for each $i<n^*$ is also unary, consequently each of the $f_i$ for $i\leq n^*$, is an identity
function. Clearly, the equality $\vec{a}(i)=\vec{b}(i)$ holds for each $i\leq n^*$.
\end{proof}

Lemma \ref{sortietype} leads to the following proposition:

\begin{prop}
If $\vec{e}\not\in\Omega$ or $\vec{e}=\langle 1, 1, 1, \dots\rangle$, then every vector space is an $\vec{e}$-Ramsey algebra. Every vector space is a
$\langle 0, 0, 0, \dots\rangle$-Ramsey algebra provided the underlying field is finite.
\end{prop}

The case where $\vec{e}\in\Omega$ is less trivial. The first proposition for such a case concerns vector spaces over finite fields. In the following
proof, we say that a finite sequence $\sigma$ is a reduction of the finite sequence $\tau$ if, for each $i\in\{0, \ldots, |\sigma|-1\}$, there exist
$f_i\in\ot(\mathcal{F})$ and a subsequence $\tau_i$ of $\tau$ such that
\begin{enumerate}
\item $\tau_0\ast\cdots\ast\tau_{|\sigma|-1}$ is a subsequence of $\tau$ and
\item $\sigma(i)=f_i(\tau_i)$ for each $i\in\{0, \ldots, |\sigma|-1\}$.
\end{enumerate}

\begin{prop}\label{longproof}
Every vector space over a finite field is an $\vec{e}$-Ramsey algebra for all $\vec{e}\in\Omega$.
\end{prop}
\begin{proof}
Fix $\vec{e}$. Let $X\subseteq A_{\vec{e}(0)}$ be given and let $\vec{b}$ be an $\vec{e}$-sorted sequence of elements of $\mathbb{F}\cup \mathbb{V}$.
Since there are infinitely many scalars in $\vec{b}$ and the field $\mathbb{F}$ is finite, there exists a scalar, call it $\rho$, which occurs
infinitely often in $\vec{b}$. The scalar $\rho$ has a finite order, meaning there exists an $s\in\mathbb{Z}^+$ such that $s\rho=0$.

We construct an $\vec{e}$-sorted reduction $\vec{a}$ of $\vec{b}$ recursively.

(Case $\vec{e}(0)=0$.) Set $\vec{a}(0)=0$. Suppose $\langle\vec{a}(0), \ldots, \vec{a}(N)\rangle$ has been constructed as a reduction of some initial
segment $\sigma$ of $\vec{b}$. If $\vec{e}(N+1)=0$, set $\vec{a}(N+1)=0$. On the other hand, if $\vec{e}(N+1)=1$, then let $\vec{a}(N+1)=v$, where $v$
is the first vector in $\vec{b}-\sigma$. Notice that $\langle\vec{a}(0), \ldots, \vec{a}(N+1)\rangle$ is a reduction of some initial segment $\vec{b}$
extending $\sigma$ (in the case $\vec{e}(N+1)=0$, we use the fact that there are infinitely many, and thus at least $s$ many, $\rho$'s in
$\vec{b}-\sigma$).

Now, since all scalar terms are $0$ in $\vec{a}$, by Facts 3 and 5 of Lemma \ref{typesoff}, it follows that we have
$\text{FR}_\mathcal{F}^{\vec{e}}(\vec{a})=\{0\}$ and the homogeneity of $\vec{a}$ for $X$ clearly holds.

(Case $\vec{e}(0)=1$.) Our plan is to start from the given $\vec{b}$ and obtain $\vec{a}\leq_\mathcal{F}\vec{c}\leq_\mathcal{F}\vec{b}$ of the same sort
$\vec{e}$ so that $\vec{a}$ is homogeneous for $X$. In the process, we will make use of an auxiliary sequence $\vec{a}'$.

The first step is, in a manner similar to the previous case, to reduce $\vec{b}$ to the sequence $\vec{c}$ having the property that its scalar terms are
all the zero scalar.

The next step is to obtain an auxiliary sequence $\vec{a}'$ homogeneous for $X$ with respect to $\{+_\mathbb{V}\}\subseteq\mathcal{F}$, obtained by
considering the subsequence of vector terms of $\vec{c}$ and applying Proposition \ref{semigroup} (recall that $(\mathbb{V}, +_\mathbb{V})$ is a
semigroup).

To obtain $\vec{a}$, set $\vec{a}(0)=\vec{a}'(0)$. Suppose $\langle\vec{a}(0), \ldots, \vec{a}(N)\rangle$ has been constructed from $\vec{c}$ as a
reduction of an initial segment $\sigma$ of $\vec{c}$ and such that the subsequence of vector terms of $\langle\vec{a}(0), \ldots, \vec{a}(N)\rangle$
forms a subsequence of $\vec{a}'$. Now, if $\vec{e}(N+1)=0$, set $\vec{a}(N+1)=0$. On the other hand, if $\vec{e}(N+1)=1$, set
$\vec{a}(N+1)=\vec{a}'(N')$ in such a way that $N'$ is the least such that $\vec{a}'(N')$ is a finite reduction of $\vec{c}-\sigma$ (possible because
$\vec{a}'$ is a reduction of $\vec{c}$) and such that the subsequence of vector terms of $\langle\vec{a}(0), \ldots, \vec{a}(N), \vec{a}'(N')\rangle$
forms a subsequence of $\vec{a}'$. This extended sequence is yet another reduction of an extended initial segment of $\vec{c}$. Thus, $\vec{a}$ is an
$\vec{e}$-sorted reduction of $\vec{c}$.

Now, all the scalar terms of $\vec{a}$ are 0, whereas its vector terms form a subsequence of $\vec{a}'$. Thus, by Fact 5 of Lemma \ref{typesoff}, if
$f\in\text{OT}(\mathcal{F})$ is vector-valued and $\vec{q}$ is a finite subsequence of $\vec{a}$, then $f(\bar{q})$ is a sum of some vectors in
$\vec{q}$. It follows by the homogeneity of $\vec{a}'$ with respect to $X$ that $f(\bar{q})\in X$ for all such $\vec{q}$ and
$f\in\text{OT}(\mathcal{F})$, or $f(\bar{q})\not\in X$ for all such $\vec{q}$ and $f\in\text{OT}(\mathcal{F})$. Therefore,
$\text{FR}_\mathcal{F}^{\vec{e}}(\vec{a})\subseteq X$ or $\text{FR}_\mathcal{F}^{\vec{e}}(\vec{a})\subseteq X^C$ by Eq. \ref{Omegafr}.
\end{proof}

The story is different when $\mathbb{F}$ is infinite. We begin with two lemmas.

\begin{lem}\label{seqinfintdom}
Suppose $(\mathbb{F}, +_\mathbb{F}, \times_\mathbb{F})$ is an infinite field. There exists a sequence $\vec{\beta}\in{^\omega\!\mathbb{F}}$ such that
for every orderly terms $f, g, f', g'$ over $\{+_\mathbb{F}, \times_\mathbb{F}\}$ and for every finite subsequences $\vec{\beta}_0\ast\vec{\beta}_1$ and
$\vec{\beta}_2\ast\vec{\beta}_3$ of $\vec{\beta}$, the following holds:
\begin{equation}\label{desiredinequation}
f(\bar{\beta}_0)+_\mathbb{F}g(\bar{\beta}_1) \\ \neq f'(\bar{\beta}_2)\times_\mathbb{F}g'(\bar{\beta}_3).
\end{equation}
\end{lem}

Such sequence $\vec{\beta}$ exists owing to the growth rate of the orderly terms over the field operations. To construct $\vec{\beta}$ recursively, we
can begin with two distinct field elements almost arbitrarily. The orderly terms that can apply to these elements are limited to the field operations
themselves. To construct the third element, we only need to pick a field element outside the image set of the first two elements under the field
operations with an added requirement, namely the element picked must satisfy Inequation (\ref{desiredinequation}). Such procedure is carried out
\emph{ad infinitum} where, at each step, the next element is picked from outside the image set of the previous elements under the orderly terms that can
operate on them, the number of available orderly terms which is finite, and explicitly requiring Inequation (\ref{desiredinequation}) to hold thus far.
Each recursive step can proceed because of the assumption that the field is \emph{infinite}. For a complete proof, see Lemma 5.4 of
\cite{teh2016ramsey}.

\begin{cor}\label{RingTeh}
No infinite field is a Ramsey algebra.
\end{cor}

The proof of this corollary hinges upon a sequence $\vec{\beta}$ given in Lemma \ref{seqinfintdom} and the associated subset $Y$ of the field
$\mathbb{F}$:
\begin{equation}\label{Y}
Y=\left\{f(\bar{\beta}_0)+_\mathbb{F}g(\bar{\beta}_1):\Psi(f, g)\right\}
\end{equation}
where $\Psi(f, g)$ is the statement ``$f, g\in\text{OT}(\{+_\mathbb{F}, \times_\mathbb{F}\})$ and
$\vec{\beta}_0\ast\vec{\beta}_1$ is a finite subsequence of $\vec{\beta}$.''

A moment's reflection reveals that, for each $f, g\in\text{OT}(\{+_\mathbb{F}, \times_\mathbb{F}\})$ and each finite subsequence
$\vec{\beta}_0\ast\vec{\beta}_1$ of $\vec{\beta}$, we have
\begin{equation*}
f(\bar{\beta}_0)\times_\mathbb{F}g(\bar{\beta}_1)\not\in Y.
\end{equation*}

Thus, given an arbitrary $\vec{a}\leq_{\{+_\mathbb{F}, \times_\mathbb{F}\}}\vec{\beta}$, we have $\vec{a}(0)+_\mathbb{F}\vec{a}(1)\in Y$ and
$\vec{a}(0)\times_\mathbb{F}\vec{a}(1)\not\in Y$. Hence, no $\vec{a}\leq_{\{+_\mathbb{F}, \times_\mathbb{F}\}}\vec{\beta}$ is homogeneous for $Y$,
whence the corollary follows.

\begin{lem}
Let $\mathbb{F}$ be infinite and $\vec{e}$ a given sort. Suppose $\vec{\beta}\in{^\omega\mathbb{F}}$ and suppose $v\in\mathbb{V}$ is a fixed
\emph{nonzero} vector. Define
\begin{displaymath}
   \vec{b}(i) = \left\{
     \begin{array}{lr}
       \vec{\beta}(i) & \text{if}\; \vec{e}(i)=0\\
       \vec{\beta}(i)\cdot v & \text{otherwise}.
     \end{array}
   \right.
\end{displaymath}
Then the following holds:

If $F$ is an orderly term over $\mathcal{F}$ and $\vec{b}_0$ is some finite subsequence of $\vec{b}$, then there exists some $|\vec{b}_0|$-ary orderly
term $f$ over $\{+_\mathbb{F}, \times_\mathbb{F}\}$ such that
\begin{displaymath}
   F(\bar{b}_0) = \left\{
     \begin{array}{lr}
       f(\bar{b}_0) & \text{if}\; \emph{Rn}(F)\subseteq \mathbb{F}\\
       f(\bar{b}_0)\cdot v & \text{if}\; \emph{Rn}(F)\subseteq \mathbb{V}
     \end{array}.
   \right.
\end{displaymath}
\end{lem}
\begin{proof}
When $\text{Rn}(F)\subseteq \mathbb{F}$, Lemma \ref{typesoff} implies that $\text{Dom}(F)$ must be the Cartesian product $\mathbb{F}^n$. We may let
$f=F$.

As for the case $\text{Rn}(F)\subseteq \mathbb{V}$, we prove the conclusion of the lemma by induction on the length of the subsequence
$\langle\vec{b}(i_1), \ldots, \vec{b}(i_n)\rangle$. Starting from $n=1$,
\begin{equation*}
F(\vec{b}(i_1))=\text{id}_\mathbb{V}(\vec{b}(i_1))=\text{id}_\mathbb{V}(\vec{\beta}(i_1)\cdot v)=\vec{\beta}(i_1)\cdot
v=\text{id}_\mathbb{F}(\vec{\beta}(i_1))\cdot v,
\end{equation*}
by Fact 1 of Lemma \ref{typesoff} and clearly $\text{id}_\mathbb{F}\in\text{OT}(\{+_\mathbb{F}, \times_\mathbb{F}\})$.

For the inductive step, let $F$ be of arity $N+1$ and $\langle\vec{b}(i_1), \ldots, \vec{b}(i_{N+1})\rangle$ some subsequence of $\vec{b}$ of length
$N+1$. Breaking up $F$ into its components gives
\begin{equation*}
F(\vec{b}(i_1), \ldots, \vec{b}(i_{N+1}))=g(h_1(\vec{b}(i_1), \ldots, \vec{b}(i_r)), h_2(\vec{b}(i_{r+1}), \ldots, \vec{b}(i_{N+1}))),
\end{equation*}
for some $g\in\mathcal{F}$ and $h_1, h_2\in\text{OT}(\mathcal{F})$. Either $g$ is vector addition or $g$ is scalar multiplication ($g$ cannot be a field
operation because $\text{Rn}(f)\subseteq\mathbb{V}$ by case assumption).

If $g$ is vector addition, then $h_1$ and $h_2$ are vector-valued. The induction hypothesis applies to give
\begin{eqnarray*}
&& F(\vec{b}(i_1), \ldots, \vec{b}(i_{N+1})) \\
&=& \left(f_1(\vec{\beta}(i_1), \ldots, \vec{\beta}(i_r))\cdot v\right)+_\mathbb{V}\left(f_2(\vec{\beta}(i_{r+1}), \ldots, \vec{\beta}(i_{N+1}))\cdot
v\right) \\
&=& \left(f_1(\vec{\beta}(i_1), \ldots, \vec{\beta}(i_r))+_\mathbb{F}f_2(\vec{\beta}(i_{r+1}), \ldots, \vec{\beta}(i_{N+1}))\right)\cdot v
\end{eqnarray*}
for some $f_1, f_2\in\text{OT}(\{+_\mathbb{F}, \times_\mathbb{F}\})$. The required $f\in\text{OT}(\{+_\mathbb{F}, \times_\mathbb{F}\})$ can be clearly
seen to be $f(x_1, \ldots, x_{N+1})=f_1(x_1, \ldots, x_r)+_\mathbb{F}f_2(x_{r+1}, \ldots, x_{N+1})$.

On the other hand, if $g$ is scalar multiplication, then $h_1$ is scalar-valued and $h_2$ is vector-valued, and the induction hypothesis applies to give
\begin{equation*}
F(\vec{b}(i_1), \ldots, \vec{b}(i_{N+1}))=\left(f_1(\vec{\beta}(i_1), \ldots, \vec{\beta}(i_r))\right)\cdot\left(f_2(\vec{\beta}(i_{r+1}), \ldots,
\vec{\beta}(i_{N+1}))\cdot v\right)
\end{equation*}
for some $f_1, f_2\in\text{OT}(\{+_\mathbb{F}, \times_\mathbb{F}\})$, from which it follows that
\begin{equation*}
F(\vec{b}(i_1), \ldots, \vec{b}(i_{N+1}))=f(\vec{b}(i_1), \ldots, \vec{b}(i_{N+1}))\cdot v,
\end{equation*}
where $f(x_1, \ldots, x_{N+1})=f_1(x_1, \ldots, x_r)\times_\mathbb{F}f_2(x_{r+1}, \ldots, x_{N+1})$ is again an orderly term over $\{+_\mathbb{F},
\times_\mathbb{F}\}$.
\end{proof}

Call a vector space nontrivial if $\mathbb{V}\neq\{O\}$.

\begin{prop}\label{nontrivialvecspaceinftyfield}
Suppose that $\vec{e}\in\Omega$. Then, no nontrivial vector space over an infinite field is an $\vec{e}$-Ramsey algebra.
\end{prop}
\begin{proof}
Let $\vec{\beta}$ be a fixed sequence guaranteed by Lemma \ref{seqinfintdom} and let $Y$ be the associated set given in Eq. \ref{Y}. Pick a nonzero
$v\in V$ and define $\vec{b}$ by
\begin{displaymath}
   \vec{b}(i) = \left\{
     \begin{array}{lr}
       \vec{\beta}(i) & \text{if}\; \vec{e}(i)=0\\
       \vec{\beta}(i)\cdot v & \text{otherwise}.
     \end{array}
   \right.
\end{displaymath}

Accompanying the set $Y$ is the set $X=\{\alpha\cdot v:\alpha\in Y\}$.

If $\vec{a}\leq_\mathcal{F}\vec{b}$, an application of the preceding lemma then shows that for each $i\in\omega$, there exist some orderly term $h_i$
over $\{+_\mathbb{F}, \times_\mathbb{F}\}$ and some finite subsequence $\vec{\beta}_0$ of $\vec{\beta}$ such that
\begin{displaymath}
   \vec{a}(i) = \left\{
     \begin{array}{lr}
       h_i(\bar{\beta}_0) & \text{if}\; \vec{e}(i)=0\\
       h_i(\bar{\beta}_0)\cdot v & \text{otherwise}.
     \end{array}
   \right.
\end{displaymath}

Fix an arbitrary $\vec{e}$-sorted reduction $\vec{a}$ of $\vec{b}$. It suffices to show that $\vec{a}$ is not homogeneous with respect to $X$ or $Y$
depending on whether $\vec{e}\in\Omega_0$ or $\vec{e}\in\Omega_1$, respectively:

(When $\vec{e}\in\Omega_1$.) First, pick $m_2>m_1>0$ such that $\vec{e}(m_1)=0$ and $\vec{e}(m_2)=1$. Let $\vec{c}_1, \vec{c}_2$ be $\vec{e}$-sorted
reductions of $\vec{a}$ such that
\begin{eqnarray*}
\vec{c}_1(0) &=& \vec{a}(0)+_\mathbb{V}\vec{a}(m_2)\\
&=& h_0(\bar{\beta}_0)\cdot v+_\mathbb{V}h_{m_2}(\bar{\beta}_1)\cdot v\\
&=& [h_0(\bar{\beta}_0)+_\mathbb{F}h_{m_2}(\bar{\beta}_1)]\cdot v,\\
\end{eqnarray*}
whereas
\begin{eqnarray*}
\vec{c}_2(0) &=& \vec{a}(m_1)\cdot\vec{a}(m_2) \\
&=& h_{m_1}(\bar{\beta}_0)\cdot [h_{m_2}(\bar{\beta}_1)\cdot v] \\
&=& [h_{m_1}(\bar{\beta}_0)\times_\mathbb{F} h_{m_2}(\bar{\beta}_1)]\cdot v. \\
\end{eqnarray*}
(Such $\vec{c}_1$ and $\vec{c}_2$ exist as $\vec{e}$-sorted reductions of $\vec{a}$ since there are infinitely many of both 0's and 1's among the terms
of $\vec{e}$: each of $\vec{c}_1-1$ and $\vec{c}_2-1$ forms a subsequence of $\vec{a}-(m_2+1)$.) It follows that
$\text{FR}_\mathcal{F}^{\vec{e}}(\vec{a})\cap X\neq\varnothing$ and $\text{FR}_\mathcal{F}^{\vec{e}}(\vec{a})\cap X^C\neq\varnothing$ since
$\vec{c}_1(0)\in X$ while $\vec{c}_2(0)\in X^C$.

(When $\vec{e}\in\Omega_0$.) Pick an $m>0$ such that $\vec{e}(m)=0$. Let $\vec{c}_1, \vec{c}_2$ be $\vec{e}$-sorted reductions of $\vec{a}$ such that
$\vec{c}_1(0)=\vec{a}(0)+_\mathbb{F}\vec{a}(m)$, while $\vec{c}_2=\vec{a}(0)\times_\mathbb{F}\vec{a}(m)$. It then follows that $\vec{c}_1(0)\in Y$ while
$\vec{c}_2(0)\not\in Y$. As such, $\text{FR}_\mathcal{F}^{\vec{e}}(\vec{a})\cap Y\neq\varnothing$ and $\text{FR}_\mathcal{F}^{\vec{e}}(\vec{a})\cap
Y^C\neq\varnothing$.
\end{proof}

We dovetail the results above into a classification of vector spaces, thus summarizing the discussion of this section:

\begin{thm}\label{summary}
Let $(\mathbb{V}, \mathbb{F}, +_\mathbb{V}, +_\mathbb{F}, \times_\mathbb{F}, \cdot)$ be a vector space.
\begin{enumerate}
\item In the case where $\mathbb{F}$ is a finite field, the vector space is an $\vec{e}$-Ramsey algebra for all sorts $\vec{e}$.
\item In the case where $\mathbb{F}$ is an infinite field, the vector space is \emph{not} an $\vec{e}$-Ramsey algebra for any sorts $\vec{e}$ except
    when $\vec{e}$ is a nonconstant eventually constant sequence or when it is constant with value $1$.
\end{enumerate}
\end{thm}

\section{A Related Homogeneous Structure}\label{hom}
Scalar multiplication in a vector space $(\mathbb{V}, \mathbb{F}, +_\mathbb{V}, +_\mathbb{F}, \times_\mathbb{F}, \cdot)$ gives rise to a collection of
unary functions given by $f_r(v)=r\cdot v$ for each scalar $r\in \mathbb{F}$ and vector $v\in \mathbb{V}$. We may, therefore, consider the algebra
$(\mathbb{V}, \mathcal{F})$, where $\mathcal{F}$ consists of the vector operation $+$ and each of the functions $f_r$. Notice that the function $f_0$
induced by the zero scalar sends every vector to the zero vector $O$. As a consequence, if $\vec{b}$ is an infinite sequence of vectors, we can
trivially obtain the infinite sequence consisting of the zero vector as a reduction. The implication of this observation is that $(\mathbb{V},
\mathcal{F})$ is trivially a Ramsey algebra.

To avoid this trivial situation, we consider the algebra $(\mathbb{V}, \mathcal{K})$, where $\mathcal{K}=\mathcal{F}\setminus\{f_0\}$. The symbol
$\mathcal{K}$ will be reserved for this class of functions in this section.

We proceed with a discussion of the orderly terms over $\mathcal{K}$. First, note that the composition $h(x, y)=+(f_r(x), f_s(y))$ is an orderly term.
The same is true if we have $n$ functions composed iteratively in this fashion. In particular, if $r_1, r_2, \ldots, r_n$ are nonzero scalars, then
\begin{eqnarray}
h(x_1, \ldots, x_n) &=& +\left(\cdots\left(+\left(+\left(f_{r_1}(x_1), f_{r_2}(x_2)\right), f_{r_3}(x_3)\right), \ldots\right), f_{r_n}(x_n)\right) \nonumber \\
&=& \sum_{i=1}^nr_i\cdot x_i \label{whatordtermslooklike}
\end{eqnarray}
is an orderly term over $\mathcal{K}$. The first lemma of this section states that orderly terms over $\mathcal{K}$ are of the form given by Eq.
\ref{whatordtermslooklike}.

\begin{lem}\label{OToverK}
$f\in\emph{OT}(\mathcal{K})$ if and only if $f(x_1, \ldots, x_n)=\sum_{i=1}^nr_i\cdot x_i$ for some nonzero scalars $r_1, \ldots, r_n$.
\end{lem}
\begin{proof}
We have seen that $\sum_{i=1}^nr_i\cdot x_i$ defines an orderly term over $\mathcal{K}$. As for the converse, we do induction on the generation of the
orderly terms over $\mathcal{K}$. Thus, first observe that the claim is clearly true for the orderly term $f_r$ of any nonzero scalar $r$; it is also
true when $f$ is the identity function since $\text{id}_\mathbb{V}=f_1$. The conclusion is equally trivial for vector addition.

Now, suppose $g\in\mathcal{K}$ and $h_1, h_2$ are orderly terms satisfying Eq. \ref{whatordtermslooklike}, i.e. for some nonzero scalars $r_i, s_j,
i\in\{1, \ldots, n\}, j\in\{1, \ldots, m\}$,
\begin{eqnarray*}\label{whatOTlookslike}
h_1(x_1, \ldots, x_n) &=& \sum_{i=1}^n r_i\cdot x_i, \\
h_2(y_1, \ldots, y_m) &=& \sum_{i=1}^m s_i\cdot y_i.
\end{eqnarray*}
If $f(x_1, \ldots, x_n, y_1, \ldots, y_m)=g(h_1(x_1, \ldots, x_n), h_2(y_1, \ldots, y_m))$, where $g$ is vector addition, then $f(x_1, \ldots, x_n, y_1,
\ldots, y_m)=\sum_{i=1}^n r_i\cdot x_i+\sum_{i=1}^m s_i\cdot y_i$, which is clearly of the desired form. On the other hand, given that $f(x_1, \ldots,
x_n)=f_r(h_1(x_1, \ldots, x_n))$ for some nonzero scalar $r$, we are led to $f(x_1, \ldots, x_n)=r\cdot\sum_{i=1}^nr_ix_i=\sum_{i=1}^n (rr_i)\cdot x_i$
and none of the scalars $rr_i$ is 0. Hence, such $f$ also assumes the form of Eq. \ref{whatordtermslooklike}. This concludes the proof of the lemma.
\end{proof}

A short discussion on linear independence is helpful. Suppose $u_1, \ldots, u_n$ are linearly independent vectors and $v=r_1\cdot u_1+\cdots+r_n\cdot
u_n$, then it is easy to check that, if $f$ is an orderly term over $\mathcal{K}$ and $f(u_1, \ldots, u_n)=v$, then $f$ is \emph{unique}, namely $f(x_1,
\ldots, x_n)=\sum_{i=1}^nr_ix_1$. We stress that this uniqueness is not met should the vectors $u_1, u_2, \ldots, u_n$ be linearly dependent.
Specifically, the zero vector $O$ is not the image of any orderly term over $\mathcal{K}$ applied to a tuple consisting of linearly independent vectors.

The first result in this section concerns the structure $(\mathbb{V}, \mathcal{K})$ for finite dimensional  vector space.

\begin{thm}
$(\mathbb{V}, \mathcal{K})$ is a Ramsey algebra for every finite dimensional vector space.
\end{thm}
\begin{proof}
Let the vector space be $n$-dimensional and let $\vec{b}\in{^\omega\!\mathbb{V}}$. The proof is by recursion; we will wave our hand on this proof.

Since $\vec{b}(0), \vec{b}(1), \ldots, \vec{b}(n)$ are $n+1$ vectors, they are linearly dependent and so for some scalars $r_0, r_1, \ldots, r_n$, not
all of which are $0$, we obtain $\sum_{i=0}^n r_i\cdot \vec{b}(i)=O$ (this need not correspond to an orderly term because, if it did, then none of the
scalars could be $0$). If $i_0<\cdots<i_k, k\leq n$, are indices corresponding to the nonzero scalars, then we have $\sum_{j=0}^k r_{i_j}\cdot
\vec{b}({i_j})=O$ as the image of $(\vec{b}(i_0), \ldots, \vec{b}(i_k))$ under the corresponding orderly term.

We repeat the above procedure on the subsequent $n+1$ vectors recursively using appropriate nonzero scalars. This generates a sequence of zero vectors
$O$ as a reduction of $\vec{b}$, which is homogeneous for any given $X\subseteq\mathbb{V}$.
\end{proof}

We now turn our attention to infinite dimensional vector spaces. To obtain our next theorem, we will be applying Corollary 4.3 of \cite{teh2016ramsey};
we state the result here for convenience.

\begin{lem}\label{corTeh}
Suppose $\mathcal{H}$ is a collection of unary operations of a set $A$, $\mathcal{G}$ a collection of nonunary operations on $A$, and $S$ is the set
$\{a\in A:f(a)=a\;\text{for each}\;f\in\mathcal{H}\}$. If $(A, \mathcal{G})$ is a Ramsey algebra, then $(A, \mathcal{H}\cup\mathcal{G})$ is a Ramsey
algebra if and only if, for each $\vec{b}\in{^\omega\!A}$, there exists $\vec{a}\in{^\omega\!A}$ such that
$\vec{a}\leq_{\mathcal{H}\cup\mathcal{G}}\vec{b}$ and $\FR_\mathcal{G}(\vec{a})\subseteq S$.
\end{lem}

For our purpose, the set $A$ stated in the lemma will be the set $\mathbb{V}$ of vectors in question and $\mathcal{H}=\{f_r:r\in\mathbb{F}, r\neq 0\},
\mathcal{G}=\{+\}$ (thus $\mathcal{K}=\mathcal{H}\cup\mathcal{G}$). It is easy to verify that the set $S=\{v\in\mathbb{V}:f(v)=v\;\text{for
all}\;f\in\mathcal{H}\}$ of fixed points of $\mathcal{H}$ is either $\mathbb{V}$ or $\{O\}$, depending on whether the underlying field is respectively
$\mathbb{F}_2$ or not.

\begin{thm}\label{infdimvsresult}
For no infinite dimensional vector space over $\mathbb{F}\neq\mathbb{F}_2$ is $(\mathbb{V}, \mathcal{K})$ a Ramsey algebra.
\end{thm}
\begin{proof}
First, note that $(\mathbb{V}, \mathcal{G})$ is a Ramsey algebra since it is a group.

We let $\vec{b}=\langle u_0, u_1, \ldots\rangle$, where $u_0, u_1, \ldots$ are linearly independent vectors. Recall that the zero vector is not the
image of any orderly term over $\mathcal{K}$ operated on a tuple of linearly independent vectors. As such, for any $\vec{a}\leq_\mathcal{K}\vec{b}$, the
element $\vec{a}(0)$, which is not the zero vector, is not in $S$ and hence $\text{FR}_\mathcal{G}(\vec{a})\not\subseteq S$ for any such $\vec{a}$.
This, by Lemma \ref{corTeh}, means that the algebra is not a Ramsey algebra.
\end{proof}

A direct proof can also be given. The intuition is to pick out a fine bit of vectors to form $X$. Again, let $\vec{b}=\langle u_0, u_1, \ldots\rangle$,
where $u_0, u_1, \ldots$ are linearly independent vectors. Recall from Lemma \ref{OToverK} that for each $f\in\text{OT}(\mathcal{K})$, we have $f(x_1,
\ldots, x_n)=\sum_{i=1}^nr_i\cdot x_i$ for some nonzero scalars $r_1, \ldots, r_n$. Let $X$ be defined by
\begin{equation}
X=\left\{v\in\mathbb{V}:\Phi(v)\right\},
\end{equation}
where $\Phi(v)$ is the statement ``$v=u_{i_0}+\sum_{j=1}^nr_j\cdot u_{i_j}$, $r_1, \ldots, r_n$ are some nonzero scalars, and $i_0<\cdots<i_n$.''

Given $\vec{a}\leq_\mathcal{K}\vec{b}$, we have $\vec{a}(0)=\sum_{j=0}^n r_j\cdot u_{i_j}$ for some $i_0<\cdots<i_n$ by Lemma \ref{OToverK}. If $r_0=1$,
then $\vec{a}(0)\in X\cap\text{FR}_\mathcal{K}(\vec{a})$, while a simple application of $f_2$ would put $f_2(\vec{a}(0))\in
X^C\cap\text{FR}_\mathcal{K}(\vec{a})$. On the other hand, if $r_0\neq 1$, then we have $\vec{a}(0)\in X^C\cap\text{FR}_\mathcal{K}(\vec{a})$, whereas
$f_{r_0^{-1}}(\vec{a}(0))\in X\cap\text{FR}_\mathcal{K}(\vec{a})$. In either case, we have $ X\cap\text{FR}_\mathcal{K}(\vec{a})\neq\varnothing$ and $
X^C\cap\text{FR}_\mathcal{K}(\vec{a})\neq\varnothing$. Thus, the conclusion of the theorem holds.

Finally, we also have the following result for vector spaces, finite or infinite, over the field $\mathbb{F}_2$.

\begin{thm}
$(\mathbb{V}, \mathcal{K})$ is a Ramsey algebra for every vector space over the field $\mathbb{F}_2$.
\end{thm}
\begin{proof}
Note that when $\mathbb{F}=\mathbb{F}_2$, we have $\mathcal{K}=\{+, f_1\}=\{+, \text{id}_\mathbb{V}\}$, whence $\text{OT}(\mathcal{K})$ is the same as
$\text{OT}(\{+\})$. The theorem follows by the fact that $(\mathbb{V}, +)$ is a semigroup, hence a Ramsey algebra (Proposition \ref{semigroup}).
\end{proof}

\section{Conclusion}\label{conc}
Much of the algebraic and analytical aspects of vector spaces have been studied and utilized ever since the notion of a vector space came into
prominence in mathematics and science. This paper offers a glimpse at the less studied combinatorics of vector spaces in the context of a combinatorial
notion which is at its infancy stage.

We wrap up with a summary of the results of Section 5 in the language of Ramsey space: The space $\mathfrak{R}^{\vec{e}}(\mathbb{V}, \mathbb{F},
+_\mathbb{V}, +_\mathbb{F}, \times_\mathbb{F}, \cdot)$ is a Ramsey space if the field $\mathbb{F}$ is a finite field; otherwise, it is a Ramsey space
only if $\vec{e}$ is constant with value $1$.

\section*{acknowledgements}

The authors gratefully acknowledge the support of the Fundamental Research Grant Scheme No.~203/PMATHS/6711464 of the Ministry of Education, Malaysia,
and Universiti Sains Malaysia.



\end{document}